\documentclass[11pt]{amsart}

 \usepackage{graphicx}
 \usepackage{amscd}

 \def\a{\alpha}
 
 \def\be{\beta}
 
 \def\de{\delta}
 
 \def\e{\varepsilon}
 \def\deta{{\dot{\eta}}}

 \def\ga{\gamma}
 \def\dga{{\dot{\gamma}}}
 \def\hga{{\widehat{\gamma}}}

 \def\vr{\varphi}
 \def\tphi{{\tilde{\phi}}}

 \def\om{\omega}

 \def\re{{\mathbb R}}

 \def\then{\Longrightarrow}
 \def\ov{\overline}
 \def\Z{{\mathbb Z}}

 \def\md{{\mathrm d}}
 
 \def\D{{\mathbb D}}
 
 \def\hG{{\widehat G}}
 \def\hF{{\widehat F}}

 \def\ff{{\mathfrak f}}
 \def\fg{{\mathfrak g}}
 
 \def\G{{\mathbb G}}
 
 \def\H{{\mathbb H}}

 \def\fh{{\mathfrak h}}
 
 \def\tH{{\widetilde H}}
 
 \def\hH{{\widehat H}}
 
 \def\K{{\mathbb K}}

 \def\cM{{\mathcal M}}
 
 \def\tM{{\widetilde{M}}}
 \def\hM{{\widehat{M}}}

 \def\on{{\ov{n}}}
 \def\oom{{\ov{m}}}

 \def\T{{\mathbb T}}

 \def\tf{{\tilde{f}}}
 \def\tu{{\tilde{u}}}
 \def\tv{{\tilde{v}}}

 \def\tq{{\tilde{q}}}
 
 \def\tq1{{\tilde{q}_1}}

 \def\hatom{{\widetilde{\omega}}}

 \def\hbe{{\widehat{\beta}}}
 
 \def\hd{{\widehat{d}}}

 \def\oH{{\ov{H}}}
 \def\oL{{\ov{L}}}

 \def\pt{\partial_t}
 \def\px{\partial_x}
 \def\py{\partial_y}

 \def \lv{\left\vert}
 \def \rv{\right\vert}
 \def \lV{\left\Vert}
 \def \rV{\right\Vert}
 \def \ov{\overline}
 
 \def \then{\Longrightarrow}

 \DeclareMathOperator{\diam}{diam}

 \DeclareMathOperator{\im}{im}

  \renewcommand{\proofname}{{\bf Proof:}}

 \theoremstyle{plain}

  \swapnumbers

 \newtheorem{Thm}{Theorem}[section]
 
 \newtheorem{Lemma}[Thm]{\bf Lemma}

 \newtheorem{Theorem}[Thm]{\bf Theorem}
 \newtheorem{Proposition}[Thm]{\bf Proposition}

 \theoremstyle{definition}

 \theoremstyle{remark}

 \newtheoremstyle{Cl}
  {5pt}
  {3pt}
  {\sl}
  {}
  {\it}
  {:}
  {.5em}
  {}

 \theoremstyle{Cl}

 \def\begincproof{
                  \renewcommand{\proofname}{\it Proof:}
                  \begin{proof}
                 }

 \def\endcproof{
                \renewcommand{\qedsymbol}{$\diamondsuit$}
                \end{proof}
                \renewcommand{\qedsymbol}{\openbox}
                \renewcommand{\proofname}{\bf Proof:}
               }

\def\bz{{\mathbf z}}

 \renewcommand{\proofname}{{\bf Proof:}}

 \title
 {Homogenization on arbitrary manifolds}

 \author[G. Contreras]{Gonzalo Contreras}

\address{CIMAT \\
          A.P. 402, 36.000 \\
          Guanajuato. GTO \\
          M\'exico.}
\email{gonzalo@cimat.mx}
\thanks{Gonzalo Contreras was partially supported by CONACYT, Mexico, grant  178838, Antonio Siconolfi was partially
supported by  FP7-PEOPLE-2010-ITN, SADCO project, GA number 264735-
SADCO}

 \author[R. Iturriaga]{Renato Iturriaga}

\address{CIMAT \\
          A.P. 402, 36.000 \\
          Guanajuato. GTO \\
          M\'exico.}
\email{renato@cimat.mx}

 \author[A. Siconolfi]{Antonio Siconolfi}

\address{Dipartimento di Matematica\\
                 Universit\`a degli Studi di Roma ``La Sapienza''\\
                  Piazzale Aldo Moro 2, 00185 Roma\\
                   Italy.}
\email{siconolf@mat.uniroma1.it}

\begin{document}

   \begin{abstract}
   We describe a setting for homogenization of convex hamiltonians
   on abelian covers of any compact manifold. In this context we also
   provide a  simple variational proof of standard homogenization
   results.
   \end{abstract}


 \maketitle

 \section{Introduction}

  \parskip +3pt

 In this paper we propose a setting in which homogenization results for
 the Hamilton-Jacobi equation which have only been obtained on the
 torus $\T^n=\re^n/\Z^n$, or equivalently in $\re^n$ with $\Z^n$-periodic conditions,
 can be carried out on arbitrary compact manifolds in a natural way.
 Moreover we show a  simple proof of the homogenization result.

 A homogenization result refers to the convergence of solutions
 $u_\epsilon $ of a problem $P_\epsilon $ with an increasingly fast variation,
 parametrized by $\epsilon$,  to a function $u_0$,
  solution of an ``averaged" problem $P_0$.

 We choose to present the simplest model of homogenization introduced in
 the celebrated paper by Lions, Papanicolaou and Varadhan~\cite{LPV}, leaving
 more sophisticated versions for future work. We believe that this setting
 will allow to translate many classical homogenization results in $\T^n$ to
 more general manifolds.

 A{\it Tonelli Hamiltonian} on the torus $\T^n=\re^n/\Z^n$ is a $C^2$ function
 $H:\re^n\times\re^n \to\re$ which is
   \begin{enumerate}
  \renewcommand{\theenumi}{\alph{enumi}}
  \item \label{na} $\Z^n${\it -periodic}, i.e. $H(x+\bz,p)=H(x,p)$, for all $\bz\in\Z^n$.
   \item \label{nb} {\it Convex:} The Hessian $\dfrac{\partial^2 H}{\partial\, p^2}(x,p)$ is positive definite
                                   for all $(x,p)$.
 \item \label{nc} {\it Superlinear:}  $\lim\limits_{p\to\infty}\dfrac{H(x,p)}{|p|}=+\infty$, uniformly on $x$.
   \end{enumerate}

 In the setting of \cite{LPV} one considers a small parameter $\e>0$
 and  the initial value problem for the Hamilton-Jacobi equation
 \begin{gather}
  \pt u^\e+ H(\tfrac x\e, \px u^\e) = 0, \label{uet} \\
 u^\e(x,0) = f_\e(x).  \label{ueti}
 \end{gather}
 If $f_\e$ is continuous  on $\re^n$ with linear growth,
 from \cite{CrLi1},  \cite{DaviZ}, \cite{Ish1}
we know that there is a unique viscosity solution of the problem
 \eqref{uet}-\eqref{ueti}.  Lions, Papanicolaou and Varadhan
 prove in \cite{LPV} that if $f_\e\to f$ uniformly  when $\e\to 0$
 and $f$ is continuous with linear growth,
 the solutions $u^\e$
 converge locally uniformly  to the unique viscosity solution of the problem \eqref{ut}-\eqref{uti}:
 \begin{gather}
 \pt u + \oH(\px u) = 0, \label{ut}   \\
 u(x,0) = f(x);  \label{uti}
 \end{gather}
 where $\oH$ is a convex hamiltonian which does not depend on $x$ and is called
 the {\it   effective Hamiltonian}. We shall see below several characterizations of
 the effective Hamiltonian $\oH$.

 Equation \eqref{uet} is seen as the Hamilton-Jacobi equation for a
 modified hamiltonian
 \begin{equation}\label{He}
 H_\e(x,p) := H(\tfrac x\e,p).
 \end{equation}
 And it is said that $H_\e\to\oH$ in the sense that the solutions of their
 Hamilton-Jacobi equations converge as stated by the homogenization result.
 The homogenization is interpreted as the convergence of solutions of Hamilton-Jacobi
 equations  when the space is {\it ``seen from
 far away''}. Alternatively, one says that the limiting problems have
  a rapidly oscillating variable $\frac x\e$ which is  {\it ``averaged''}
by the homogenization limit.

 The effective Hamiltonian $\oH$ is usually highly
 non-differentiable, but the solutions of the problem \eqref{ut}-\eqref{uti} are easily
 written because the characteristic curves for the equation \eqref{ut}   are the straight lines,
 and $p=\partial_xu$ is constant along them. Thus
 \begin{equation}\label{uoL}
 u(y,t) =  \min_{x\in\re^n}\left\{f(x) + t\,\oL\left(\tfrac{y-x}t\right) \right\},
 \end{equation}
 where
 \begin{equation}\label{oL}
 \oL(v) := \max_{p\in\re^n} \big[\,p(v) -\oH(p)\,\big]
 \end{equation}
is the {\it effective Lagrangian}.
The simplicity of this limit solution and
in general the possibility of using coarse grids in numerical analysis
for the averaged problem
are the main advantages of the homogenization
in applications.

 It turns out that the effective Lagrangian $\oL$ is  Mather's minimal action functional $\beta$
  and the effective Hamiltonian
  $\oH$ is  Mather's alpha function, the convex
  dual of  $\beta$.
 We recall the construction of the beta function below in \eqref{beta}.
 For now, it is interesting to observe that $\beta=\oL$ is defined in the homology group
 $H_1(\T^n,\re)$, but in \eqref{uoL} it is applied on velocity vectors.
 Similarly $\alpha=\oH$ is defined in the cohomology group $H^1(\T^n,\re)$
 and in equation \eqref{ut} it is applied  on gradients $\partial_x u$.
 This can be done because in the case of the torus $\T^n$ the groups
 $H_1(\T^n,\re)\approx \re^n\approx H^1(\T^n,\re)$ can be identified with
 the second factor in the trivial bundle $T\T^n=\T^n\times\re^n$.

 Mather's alpha function $\a=\oH$ is related with Ma\~n\'e's critical value \cite{Ma7}
 by\footnote{Here
 $(L-P)(x,v):=L(x,v)-P\cdot v$}
 \linebreak
 $\a(P)=c(L-P)$.
 As such, the effective Hamiltonian $\oH=\a$
   has several interpretations, see \cite{CILib}, we name some of them:
 \begin{enumerate}
 \renewcommand\theenumi{\roman{enumi}}
 \item  $\a$ is the convex dual of $\beta$.
 \item  $\a(P) = \inf\{\, k\in\re \;|\; \oint_\ga (L-P+k) \ge 0\quad \forall\text{ closed curve $\ga$ in $\T^n$}\;\}$.
 \item  $\a(P) = \inf\{\, k\in \re\;|\; \Phi_k > -\infty\,\}$, where $\Phi_k:M\times M\to\re$ is
 $$
 \Phi_k(x,y) = \textstyle\inf\big\{\, \oint_\ga (L-P+k) \; \big| \; \ga \text{ curve in $\T^n$ from $x$ to $y$ }\big\}.
 $$
 \item $\a(P) = - \inf\big\{ \int (L -P) \;d\mu\;\big|\;   \mu \text{ is an invariant measure for $L$ }\}.   $
 \item  $\a(P)$ is the energy of the invariant measures $\mu$ which minimize $\int (L-P)\;d\mu$.
 \item \label{vi} From Fathi's  weak KAM theory \cite{Fa1}, \cite{FathiBook},
 $\a(P)$ is the unique constant for which there are
 global viscosity solutions of the Hamilton-Jacobi equation
 $$
 H(x,P+\partial_xv)=\a(P), \qquad x\in\T^n.
 $$
 \item\label{vii}
 From \cite{CIPP}, $\a(P)=\displaystyle\min_{u\in C^1(\T^n,\re)}\;\max_{x\in \T^n} \;H(x,P+\partial_xu)$.
 \item From \cite{CIPP},
          $\a(P)$ is the  minimum of energy levels which contain a lagrangian graph in $T^*\T^n$
          with cohomology class $P$.
 \end{enumerate}

\bigskip

 Following~\cite{LPV},   knowing that the homogenization theorem holds, it is easy to prove that $\oH=\a$.
  Indeed,
 consider the special case of  affine initial conditions,
 $a\in\re^n$, $P\in( \re^n)^*$,
 \begin{equation}\label{il}
 f(x) = u(x,0) = a + P\cdot x.
 \end{equation}
 The solution of
 \begin{equation}\label{HJa}
 u_t+\a(\partial_xu)=0
 \end{equation}
 with initial condition \eqref{il} is
\begin{equation}\label{linsol}
u(x,t) = a + P\cdot x - \a(P)\, t.
\end{equation}
Using \eqref{vi}, fix a $\Z^n$-periodic viscosity solution of the Hamilton-Jacobi equation
 \begin{equation}\label{cell}
 H(x,P+\partial_xv) = \a(P)
 \end{equation}
  given by the weak KAM theorem in \cite{Fa1}.
 Let
  \begin{align}\label{ue}
 u^\e(x,t) &= u(x,t) + \e\;v\big(\tfrac x\e\big).
 \end{align}
 Define $f_\e$ by
  $u^\e(x,0)=f_\e(x)$. Then  $u^\e$ solves the problem \eqref{uet}--\eqref{ueti}.
 Since $f_\e\to f$ and $u^\e\to u$ uniformly, equation \eqref{HJa}
 must be the Hamilton-Jacobi equation for the effective Hamiltonian and hence
  $\oH(P)=\a(P)$.

 \bigskip

  Let $M$ be a compact path--connected riemannian manifold without boundary. 
  A {\it Tonelli hamiltonian} on $M$
 is a $C^2$ function $H:T^*M\to \re$ on the cotangent bundle $T^*M$ which is
 convex and superlinear as in \eqref{nb}, \eqref{nc} above. We want to generalize the
 Lions-Papanicolaou-Varadhan Theorem  to this setting.
 The generalization of their theorem to other compact manifolds has three
 problems, namely

 \begin{enumerate}
 \item \label{p1} It is not clear how to choose the generalization of $\tfrac x\e$.
 \item[(2)] In the modified hamiltonian $H_\e$ in \eqref{He} the base point
           changes to $\tfrac x\e$ but the moment $p$ {\it ``remains the same''}.
           It is not clear how to do this in non-parallelizable manifolds.
           Similarly, the effective Hamiltonian $\oH(P)$ {\it ``does not depend on $x$''}.
           Again, this is not natural if the manifold is not parallelizable.
          \stepcounter{enumi}
 \item\label{p3} Mather's alpha function, the candidate for the effective Hamiltonian,
 is defined on the first cohomology group $\a:H^1(M,\re)\approx\re^k\to\re$,
 which may not be a cover of the manifold. Thus the (limiting) effective Hamiltonian
 and the Hamilton-Jacobi equations for $H_\e$ would be defined on very
 different spaces. In particular, these spaces usually have different dimensions.
 \end{enumerate}

 To solve the last problem we will use an ad hoc definition of convergence of spaces
 very much inspired  by the Gromov Hausdorff convergence.
 For the second problem, a change of variables in the torus allows to change
 the parameter $\e$  in the space variables  $\tfrac x\e$ to the momentum variables.
 Indeed,  write
 \begin{equation}\label{ue0}
 u^\e(x,t) = v^\e\big(\tfrac x\e,t\big).
 \end{equation}
 Then the problem \eqref{uet}-\eqref{ueti} for $v^\e(y,t)$ becomes
 \begin{gather}
 \pt v^\e + H(y,\tfrac 1\e \py v^\e) = 0; \label{ve1} \\
 v^\e(y,0) = f_\e(\e y).  \label{ve2}
 \end{gather}
 Observe that now equation \eqref{ve1} makes sense in any manifold,
 but equation \eqref{ve2} does not. We will take care of this later.

  Given a metric space  $(\cM,\md)$, a family of metric spaces  $(\cM_n,\md_n)$ and continuous maps
  \linebreak
  $F_n:(\cM_n,\md_n)\to (\cM, \md)$, we  give a notion of convergence of $\cM_n$ to $\cM$ through $F_n$, $F_n$ should be
  interpreted as a
   telescope through  which we look at the limit space $\cM$ from $\cM_n$. The definition basically
    requires equivalence of
   distances $\md_n$ and the pullbacks  of $\md$ under $F_n$, with adjustments  to cope  with the case where
   $F_n$ is not injective and/or not surjective.
   We will  say that $\lim_n (\cM_n,\md_n,F_n) = (\cM,\md)$ if
  \begin{enumerate}
   \renewcommand{\theenumi}{\alph{enumi}}
  \item \label{c1}
  There are $K> 1$ and $A_n>0$ such that $\lim_nA_n=0$ with
     \[
  \forall x,y\in \cM_n,\qquad
  K^{-1}\,\md_n(x,y)-A_n \le \md\big(F_n(x),F_n(y)\big)\le K\,\md_n(x,y).
  \]
 \item \label{c2} For any $y\in \cM$ there is a sequence $x_n\in \cM_n$ such that
             $\lim_n F_n(x_n)=y$.
  \end{enumerate}

  Observe that condition~\eqref{c2} implies that for any  metric ball $\K$ of $\cM$
  \begin{equation}\label{ball}
    \K \cap F_n(\cM_n) \neq \emptyset \quad\hbox{for $n$
    sufficiently large,}
\end{equation}
a kind of surjectivity statement. Similarly, using the convention $\diam\emptyset =0$,
condition~\eqref{c1} implies that
$\lim_n \diam F_n^{-1}\{y\} = 0$ for all $y\in M$.
\smallskip

 If $\lim_n(\cM_n,\md_n,F_n)= (\cM,\md)$ and $f_n: \cM_n\to \re$, $f:\cM\to\re$, 
 we say that $f_n$ {\it locally  uniformly $F_n$--
converges } to $f$ if for every $x \in \cM$ and subsequences $y_{n_k} \in \cM_{n_k}$
 with $\lim_{n_k} F_{n_k}(y_{n_k})= x$,  one has $\lim_{n_k}f_{n_k}(y_{n_k})=f(x)$.

If $\lim_n(\cM_n,\md_n,F_n)= (\cM,\md)$ and $f_n: \cM_n\to \re$, $f:\cM\to\re$, we say that
  $f_n$ {\it uniformly $F_n$--converges to} $f$ if it locally uniformly $F_n$--converges to $f$
   and in addition
  \[  \lim_n \sup_{x\in \cM_n} \lv f_n(x) -f(F_n(x))\rv =0.\]

  \medskip

  The initial Hamiltonian will be the lift
  of $H$ to the maximal free abelian cover of $M$ defined as follows.
   Let $\tM$ be the covering space of $M$ defined by
 $\pi_1(\tM) = \ker {\fh}$, where
 $\fh:\pi_1(M)\to H_1(M,\Z)$
 is the Hurewicz homomorphism. Its group of deck transformations is
 $\G= \text{im}[\pi_1(M)\to H_1(M,\Z)]$, which is a free abelian group,
 $\G\approx \Z^k\subset H_1(M,\re) \approx \re^k$.
 Observe that the large-scale structure of the covering space $\tM$
 is given by $\G=\Z^k$. In the homogenized problem the position space, or configuration space, is
the homology group $x\in H_1(M,\re)\approx\re^k$, and the momenta, and the derivatives $\partial_xu$,
are in its dual, the cohomology group $\{p, \, \partial_xu\}\subset H^1(M,\re)\approx \re^k$.

  Let $d$ be the metric on $\tM$ induced by the lift of the riemannian metric on $M$, set for any $\e >0$,
   $d_\e:=\e d$.
 Then $(\tM, d_\e)$ converges to $H_1(M,\re)$ in the same way as
 $\e\Z^k$ converges to $\re^k$ or $\e\,\G\overset{\e}\longrightarrow H_1(M,\re)$.

 To be precise,
 fix a basis $c_1,\ldots, c_k$ of $H^1(M,\re)$ and fix closed 1-forms $\om_1,\ldots,\om_k$ in $M$
 such that $\om_i$ has cohomology class $c_i$. By the universal coefficient theorem
 $H_1(M,\re)=H^1(M,\re)^*$.
 Let $G:\tM\to H_1(M,\re)$ be given by
 \begin{equation}\label{defG}
    G(x)\cdot \Big(    \sum_i a_i \, c_i\Big) = \oint_{x_0}^x\Big(  \sum_i a_i \,\hatom_i\Big),
\end{equation}
 where $x_0$ is a base point in $\tM$ and $\hatom_i$ is the lift of $\om_i$ to $\tM$.
 Since $\hatom$ is exact, the integral does not depend on the choice of the path
 from $x_0$ to $x$. Notice that the function $G$ depends on the choice of $x_0$,
 on the choice of the basis $\{c_i\}$ and of representatives  $\om_i$ in $c_i$. However, we shall be interested in the functions $\e G$ for $\e$ small, and these
 dependencies disappear in the limit $\e\to 0$. On the points in $\tM$ in the same fiber as
 $x_0$ the value of $G$ does not depend on  the chosen basis $\{\om_i\}$ because
 in that case it is an integral on a closed curve in $M$. In general, if $x$, $y$ belong to the same fiber, then $G(x)- G(y)$
 is the  transformation in  $\G= \text{im}[\pi_1(M)\to H_1(M,\Z)]$ carrying $y$ to $x$, condition which uniquely identifies it,
 conversely, any deck transformation admits a representation of this type. We set $F_\e= \e \, G$.

 \begin{Proposition} \label{Pspaces}
 For the maximal free abelian cover we have that
 $$
 \lim_\e (\tM,d_\e,F_\e) = H_1(M,\re).
 $$
 \end{Proposition}

 We finally  address the problem of initial conditions. If we want to mimic  the torus case where the same initial datum
 can be taken for approximating
 as well  as for limit equations,  we can think of transferring  a continuous datum $f$ defined in $H_1(M, \re)$ to $\tM$ by
 setting
 $f_\e(y)=f(F_\e(y))$, which is interpreted as $f$ ``seen'' on $(\tM, d_\e)$. We recognize in the relationship between
 $f$ and $f_\e$
 a special instance of uniform $F_\e$--convergence, and this is indeed the way  in which
 the matter will be presented in the forthcoming statement of the main
 homogenization result. The proof of this theorem will be given in Section \ref{STA}.

 \begin{Theorem}\quad\label{TA}

 Let $H:T^*M\to\re$ be a Tonelli hamiltonian on a compact manifold $M$.
 Let
 \linebreak
 $f_\e:\tM\to\re$, $f: H_1(M,\re) \to \re$  be  continuous  with $f$ of at most linear growth, assume that
 $f_\e$   uniformly $F_\e$--converges   
  to $f$.
 Let $\tH:T^*\tM\to\re$ be the
 lift of $H$ to $\tM$.
 Let $v^\e$ be the  viscosity solution to the
 problem
 \begin{gather}
 \phantom{\qquad x\in \tM,\; t>0.}
 \partial_t v^\e + \tH(x,\tfrac 1\e\partial_x v^\e) = 0, \qquad x\in \tM,\; t>0;
 \label{vve1}\\
 v^\e(x,0) = f_\e(x).
 \label{vve2}
 \end{gather}

 Then the family of functions $v^\e:\tM_\e\times[0,+\infty[\to\re$
 locally uniformly  $F_\e$--converges  in $ \tM \times ]0,+ \infty[$
 \footnote{Our proof of the uniformity does not extend to $t=0$.}
  to the viscosity solution
 $u:H_1(M,\re)\times[0,+\infty[\to\re$ of the problem
\begin{gather}
\phantom{\qquad h\in H_1(M,\re),\; t>0.}
\partial_t u + \oH(\partial_h u) = 0, \qquad h\in H_1(M,\re),\; t>0;
\label{uue1}\\
u(h,0) = f(h);
\label{uue2}
\end{gather}

where the effective Hamiltonian $\oH$ is Mather's alpha function $\oH=\a:H^1(M,\re)\to\re$.

 \end{Theorem}

 Several comments are still in order:

 1) The convergence destroys the differential structure of the spaces.
 Nevertheless we  obtain convergence of solutions $v^\e$ to a solution
 of a partial differential equation on the limit space because the Hamilton-Jacobi
 equation is an encoding of a variational principle. Namely, its solutions
 are the minimal cost functions under the Lagrangian. This variational
 principle is preserved under the limit of spaces.

2) In this setting it is possible to prove the homogenization theorem
     using standard methods.
    This is very good news since we  expect that many homogenization results
     generalize from the torus to other manifolds.
However, using a result of Mather, we will provide another proof,
which is essentially a change of variables in the Lax formula.

3) Motivated by possible applications we extend the result to other  Abelian covers.

 \begin{figure}[h]\label{cover}
\resizebox*{11cm}{9cm}{\includegraphics{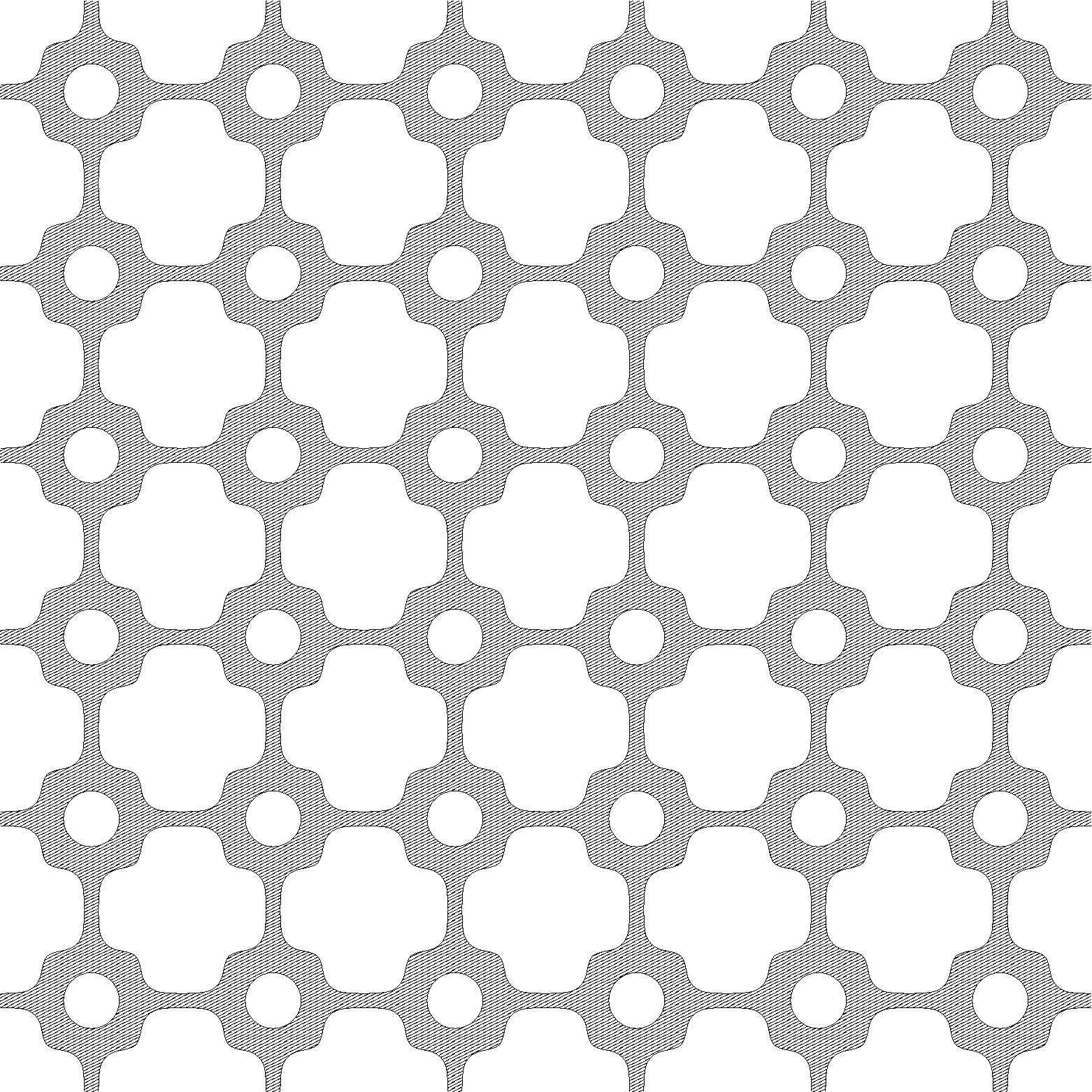}}
 \caption{A free abelian (sub)cover of the compact orientable surface of genus 3 with
group of covering transformations $\Z^2$.}
\end{figure}

\subsection{Subcovers}\quad

For general abelian covers, i.e. covering spaces whose group of deck transformations $\D$ is abelian, the torsion part of $\D$ is killed under the
limit of $\hM_\e=(\hM,\e\hd)$. Thus the limit is the same as for a free abelian
cover, where $\D$ is free abelian. These coverings are subcovers of the
maximal free abelian cover $\tM$. In this case we have similar results as
in Theorem~\ref{TA}.

Let $L:TM\to\re$ be the lagrangian of $H$ i.e.
 \begin{equation*}
 L(x,v) := \max_{p\in T^*_xM} \big[\, p(v) - H(x,p)\,].
 \end{equation*}
 The Euler-Lagrange equation for $L$ is
 \begin{equation}\label{EL}
 \tfrac d{dt} \partial_v L = \partial_x L.
 \end{equation}
 It determines a complete flow $\vr$ on $TM$ by $\vr_t(x,v)=(\ga(t),\dga(t))$ where
 $\ga$ is the solution of \eqref{EL} with $(\ga(0),\dga(0))=(x,v)$.
 Given an invariant Borel probability $\mu$ for $\vr_t$ with compact support
 define its {\it homology class} $\rho(\mu)\in H_1(M,\re) = H^1(M,\re)^*$ by
 $$
 \rho(\mu) \cdot c = \int_{TM} \om \, d\mu,
 $$
 where $\om$ is any closed 1-form on $M$ with cohomology class $c$.
 Mather's minimal action function is $\be:H_1(M,\re)\to\re$,
 \begin{equation}\label{beta}
 \be(h) :=\inf_{\rho(\mu)=h}\int L\;d\mu,
 \end{equation}
 where the infimum is over all $\vr_t$-invariant probabilities with homology $\rho(\mu)=h$.

   Free abelian covers $\hM$ are obtained from  normal subgroups of the fundamental group  $\pi_1(M)$
   containing the commutators. If  $\fg$ is the canonical epimorphism from  $\pi_1(M)$ to the quotient group
   then  $\pi_1(\hM)=\ker \fg$ and the group of deck transformations  is $\im \fg$.
   It can be identified to $\Z^\ell$, up to the choice of a basis.
   Since $H_1(M,\Z)$ is
    the abelianization of $\pi_1(M)$, such
  epimorphism $\fg$ factors  as $\fg=\ff\circ \fh$ with
    $\fg:\pi_1(M)\overset \fh \longrightarrow H_1(M,\Z)
  \overset{\ff}{\longrightarrow} \im \fg$.
  The linearization of $\ff$ gives a linear epimorphism
  $\ff:H_1(M,\re)\to \im \ff \approx \re^\ell$.  Loosely speaking, the elements  of $\im \ff$ can be interpreted as homology
  classes adapted to the cover $\hM$.  We denote by $\pi$ the covering projection of $\tM$ onto $\hM$.
  We record for later use that, given any norm on $H_1(M,\re)$,
  $H_1(M, \Z) \cap \ker \ff \approx \Z^{k-\ell}$  is $B$--dense in $\ker \ff \approx \re^{k-\ell}$  for a suitable constant $B$.
  Further, for any fixed $x_0 \in \tM$
\begin{equation}\label{ff}
    H_1(M, \Z) \cap \ker \ff= \{G(x)-G(x_0) \mid \pi(x)= \pi(x_0)\}.
\end{equation}
The metric on $\hM$, denoted by $\hd$, induced by the lift of the riemannian metric on $M$,  is given by
\[ \hd(x,y) = \inf \{d(z,w) \mid z \in \pi^{-1}(x), \, w \in \pi^{-1}(y)\}\]

\smallskip

The minimal action functional
  for the cover $\hM$ is
  $\hbe:\re^\ell=\im\ff \to \re$,
  \begin{equation}\label{hbe}
  \hbe(z) = \inf\{\, \be(h) \,|\, \ff(h)=z\,\}.
  \end{equation}
  This can also be interpreted as the average action of minimizing
  Euler-Lagrange orbits on $\hM$ with asymptotic direction $z$.
  The  effective Hamiltonian for $\hM$ is $\ov{H}=\hbe^*$, the convex dual
  of $\hbe$:
  \begin{align}\label{}
  \ov{H}(p) =\hbe^*(p) &= \max_{z\in\im \ff} p(z) - \hbe(z) \notag \\
  &= \max_{h\in H_1(M,\re)} p(\ff(h)) - \be(h) \notag \\
  &= \a(\ff^*(p)), \label{f*a}
  \end{align}
 where $\ff^*$ is the homomorphism $\ff^*:(\im \ff)^*\to H_1(M,\re)^*=H^1(M,\re)$
 induced by $\ff$.

 By \eqref{ff}
 \begin{equation}\label{abe1}
    \ff(G(y))= \ff(G(x)) \qquad\hbox{whenever $\pi(x)=\pi(y)$}.
\end{equation}
Then $\ff (G(\cdot))$ is the lift to $\tM$ of a map from $\hM$ to $\im \ff$ that we denote by $\hG$. We set, for any $\e >0$,
 $\hF_\e= \e \, \hG$.

\begin{Proposition}\label{SubcoverSpaces}
 $$
 \lim_\e(\hM,\e \hd ,\e \hF_\e)= \im \ff.
 $$
 \end{Proposition}
 The proof will be given in the next section. We proceed to give the statement
 of the homogenization result  for general abelian covers, the proof will be provided in Section \ref{subcovers}.

 \begin{Theorem}\quad\label{TB}

 Let $f_\e:\hM \to\re$, $f: \im \ff \to \re$ be continuous with $f$ of linear growth.
  Assume that
 $f_\e$  uniformly $\hF_\e$--converges to $f$.
 Let $\hH:T^*\hM\to\re$ be the
 lift of  a Tonelli Hamiltonian $H$, defined on $M$, to $\hM$.
 Let $v^\e$ be the viscosity solution to the
 problem
 \begin{gather}
\phantom{\qquad x\in \hM,\; t>0. } \partial_t v^\e + \hH(x,\tfrac 1\e\partial_x v^\e)
= 0, \qquad x\in \hM,\; t>0; \\
 v^\e(x,0) = f_\e(x).
 \end{gather}
 Then the family of functions $v^\e:\hM_\e\times[0,+\infty[\to\re$
 locally uniformly $\hF_\e$--converges in $ \hM \times ]0,+ \infty[$
to the solution $u:\im \ff\times[0,+\infty[\to\re$ of the problem
\begin{gather}
\phantom{\qquad q\in \im \ff,\; t>0. }
\partial_t u + \oH(\partial_q u) = 0, \qquad q\in \im \ff,\; t>0;
\label{lim1}
\\
u(q,0) = f(q);
\label{lim2}
\end{gather}
where the effective Hamiltonian $\oH:(\im \ff)^*\to\re$, is $\oH=\ff^*\a$ given by \eqref{f*a}.

\end{Theorem}

 \section{Convergence of Spaces}\label{spaces}

In this section we provide proofs for Propositions \ref{Pspaces}, \ref{SubcoverSpaces},
and show that some basic  properties
holding for usual uniform convergences are still true in our setting.

\smallskip

 \noindent {\bf Proof of proposition~\ref{Pspaces}:}

 Observe that for the finite dimensional space $H_1(M,\re)$ we can use any norm $\|\cdot\|$.

 If $\om=\sum_ia_i \,\om_i$  and $\lV \om\rV:=\sup_{x\in M}|\om(x)|$,
 \begin{align*}
 \Big\vert \big[G(x)-G(y)\big]\cdot [\om]\Big\vert = \lv\oint_x^y\widetilde\om\; \rv
 \le \lV\om\rV\, d(x,y).
 \end{align*}
 Then there is $K_0> 0$ such that
 \begin{equation}\label{GLip}
 |G(x)-G(y)|\le {K_0}\; d(x,y),
 \end{equation}
 and using that $F_\e= \e\,G$ and $d_\e=\e\,d$, we have that
  \begin{equation}\label{FLip}
 \| F_\e(x)-F_\e(y)\| \le K_0\;d_\e(x,y).
 \end{equation}

 Write $\G=\im [\pi_1(M)\to H_1(M,\re)]$, the group of covering transformations of $\tM$.
 Fix $x_0\in \tM$.
 Let $e_1,\ldots,e_k$ be a basis of $\G$ and let $\hga_i$
 be a minimal geodesic from $x_0$ to $e_i(x_0)=:x_0+e_i$.
 If $\pi:\tM\to M$ is the projection, the concatenation
 $(\pi\circ\hga_1)^{n_1}*\cdots*(\pi\circ\hga_k)^{n_k}$ lifts to
 a curve from $x_0$ to $x_0+\on$, where $\on = \sum n_i e_i$.
 Let  $\ell_i$ be the length of $\ga_i$. Then
 \begin{align*}
 d(x_0,x_0+\on)\le
 \textstyle{\sum_i}\; n_i \,\ell_i
 \le (\max_i \ell_i) \,k\, \|\on\|
 =:K_2\,\|\on\|.
 \end{align*}
 For any $\on, \oom\in \G\approx \Z^k$, we have
 \begin{align*}
 d(x_0+\on,x_0+\oom) = d(x_0,x_0+(\on-\oom))
 \le K_2\, \|\on-\oom\|.
 \end{align*}

 If $x, y\in \tM$ there are two elements $x_0+\on$, $x_0+\oom$ of the orbit of $x_0$ such that
 \linebreak
 $d(x,x_0+\on)\le D$ and $d(y,x_0+\oom) \le D$, where $D:=\diam M$.
 We have that
 \begin{align*}
 d(x,y)&\le d(x,x_0+\on)+d(x_0+\on,x_0+\oom)+d(x_0+\oom,y)
 \\
 &\le K_2\, \|\on-\oom\| + 2D.
 \end{align*}

 Observe that
 \begin{equation}\label{DG}
 G(x_0+\oom)-G(x_0+\on) = \oom-\on \in H_1(M,\re).
 \end{equation}
 Using the Lipschitz property \eqref{GLip} for $G$,
 \begin{align*}
 \lv G(x)-G(y)\rv &\ge \|\oom-\on\| - \|G(x) -G(x_0+\on)\| -\|G(y)-G(x_0+\oom)\|
 \\
 &\ge \|\oom-\on\| - 2\, K_0 \, D.
 \end{align*}

 Therefore
 $$
 d(x,y) \le K_2\,\|G(x)-G(y)\| + 2\, K_0\, K_2 \,D + 2D.
 $$
 For $A:= 2 K_0 D + 2 K_2^{-1} D$,
 $$
 \forall x,y \in \tM, \qquad  K_2^{-1} \, d(x,y) - A \le \|G(x)-G(y)\|.
 $$
 Multiplying the inequality by $\e$, we get
  \begin{equation}\label{FLow}
 \forall x,y \in \tM, \qquad  K_2^{-1} \, d_\e(x,y) - \e\,A \le \|F_\e(x)-F_\e(y)\|.
 \end{equation}
 Inequalities \eqref{FLip} and \eqref{FLow} prove condition \eqref{c1} of the convergence.

 Condition \eqref{c2} follows from the fact that the image of the $\G$-orbit of $x_0$,
 $$
 F_\e(x_0+\G) = \e \G = \e \Z^k \subset \re^k = H_1(M,\re),
 $$
 is $\e$-dense in $H_1(M,\re)$.
\qed

\medskip

\noindent {\bf Proof of proposition~\ref{SubcoverSpaces}:}

 We endow $\im \ff$ with
the norm $\|\cdot\|_\ff$ defined  as
\[ \|q\|_\ff = \min\{\|h\| \mid h  \in H_1(M,\re) \;\;\hbox{with}
\;\; \ff(h)= q\},\]
where $\|\cdot\|$  is any norm for $H_1(M,\re)$.

Let $x$, $y$ be in $\hM$ and $z$, $w$ any element in
$\pi^{-1}(x)$, $\pi^{-1}(y)$, respectively. Exploiting the
convergence  proved in Proposition \ref{Pspaces}, we have
\[
    \|\hG(x)-\hG(y) \|_\ff = \|\ff( G(z) - G(w)) \|_\ff \leq \| G(z) - G(w)\| \leq k \, d(z,w) ,
\]
which implies, by  arbitrariness of $z$, $w$
\begin{equation}\label{subco1}
    \|\hG(x)-\hG(y) \|_\ff \leq  k \, \hd(x,y).
\end{equation}

By $B$--density of $H_1(M,\Z) \cap \ker \ff$ in $\ker \ff$ and \eqref{ff}, we find $z \in \pi^{-1}(x)$, $w \in \pi^{-1}(y)$ with
\[\|\hG(x)-\hG(y) \|_\ff \geq \|G(z)-G(w)\| - B\]
from this we derive, again using Proposition \ref{Pspaces}
\[\|\hG(x) - \hG(y) \|_\ff \geq  k^{-1} \, d(z,w) - A -B  \geq k^{-1} \, \hd(y,x)  - A  -B. \]
The last inequality together with \eqref{subco1} gives, up to multiplying by $\e$, the first property of the convergence. Given
 $q = \ff(h) \in \im \ff$, we know that there is a sequence $x_\e$ in $\tM$ with $\e \, G(x_\e)$ converging to $h$, therefore
 \[ q = \lim_\e \e \, \ff( G(x_\e)) = \lim_\e \e \, \hG(\pi(x_\e)).\]
 This concludes the proof.
\qed

\smallskip

We go back to the abstract setting considering general metric spaces $\cM$, $\cM_n$,  we will assume from now on
the limit space $\cM$ to be locally compact, namely we require
\begin{enumerate}
  \item[]  All the metric balls of $\cM$ are relatively compact.
\end{enumerate}
This will fit the frame of homogenization.
To avoid pathological cases, we also make clear that the compact subsets  of $\cM$ we will consider in what follows,
usually denoted by $\K$, are understood,
without further mentioning, to be with nonempty interior. This will guarantee that $F_n^{-1}(\K) \neq \emptyset$, at least for
$n$ large, thanks to \eqref{ball}.

If $\lim_n(\cM_n,\md_n,F_n)= (\cM,\md)$ and $f_n: \cM_n\to \re$, we say that the family  $f_n$ is
{\it $F_n$--locally equicontinuous} if for any
compact subset  $\K$ of $\cM$ and any $ \e >0$, there exists  $\de=\de(\e,\K)>0$
   such that
   $$ x,y\in F_n^{-1}(\K),\quad
   \md_n(x,y)<\de\quad  \then \quad \lv f_n(x)-f_n(y)\rv<\e.
   $$

\smallskip

In accord to what  the term uniform suggests, continuity is stable   under the locally uniform $F_n$--convergence.

\begin{Proposition} \label{ascoli1}
Assume $\lim_n(\cM_n,\md_n,F_n)= (\cM,\md)$, and take $f_n: \cM_n\to \re$, $f:\cM\to\re$ with $f_n$ locally uniformly
$F_n$--convergent to $f$. If all the $f_n$ are continuous then $f$ is continuous.
\end{Proposition}

\begin{proof} We claim that the $f_n$ are $F_n$--locally equicontinuous. In fact, if this were not true there should be,
taking into account that all the $f_n$ are continuous, for a
 given compact $\K$ in $\cM$ and $\de >0$, (sub)sequences $z_n \in F_n^{-1}(\K)$, $y_n \in F_n^{-1}(\K)$  with
\begin{equation}\label{ascoli11}
    \md_n(z_n,y_n) \to 0 \quad\hbox{and} \quad |f_n(y_n)-f_n(z_n)| > \de.
\end{equation}
By condition \eqref{c1}  in the definition of spaces convergence $\md(F_n(x_n),F_n(z_n)) \to 0$, and so,
 owing to  compactness of $\K$, $F_n(x_n)$ and $F_n(z_n)$ converge, up to subsequences, to the same element, say $x$,
 of $\cM$. By the very    definition of local uniform convergence, we deduce
 $\lim_n f_n(x_n)=\lim_n f_n(z_n)=f(x)$, which is in contrast with \eqref{ascoli11}.

 We proceed proving that $f$ is uniformly continuous in $\K$.
Given $\e>0$ let $\de=\de(\e, \K)>0$ be such that
$$
\forall  n,\quad \forall y, z \in F_n^{-1}(\K),\quad
\md_n(y,z)<\de \then |f_n(y)-f_n(z)|<\e.
$$
 Let $x_0,x_1\in \cM$ with $\md(x_0,x_1)< \frac \de K$, where $K$ is the constant appearing in condition \eqref{c1}
 of the definition of spaces
 convergence. Let $z_n, y_n\in \cM_n$ with $\lim_n F_n(z_n)=x_0$, $\lim_n F_n(y_n)=x_1$, then
$\lim_n f_n(z_n)=f(x_0)$, $\lim_n f_n(y_n)=f(x_1)$
and
$$
\md_n(z_n,y_n) \le K \, \md(F_n(z_n),F_n(y_n)) + A_n\, K
\overset{n}\longrightarrow K \, \md(x_0,x_1)<\de.
$$
Thus
$$
|f_n(z_n)-f_n(y_n)|<\e.
$$
Taking $\limsup_n$ on the inequality
$$
|f(x_0)-f(x_1)| \le |f(x_0)-f_n(z_n)|+|f_n(z_n)-f_n(y_n)|+|f_n(y_n)-f(x_1)|,
$$
we obtain that
$$
|f(x_0)-f(x_1)|\le\e,
$$ as desired.
\end{proof}

 If $\lim_n(\cM_n,d_n,F_n)=(\cM,d)$, we say that a family of functions
 $f_n:(\cM_n,d_n)\to\re$ {\it converges pointwise} to $f:(\cM,d)\to\re$ if
 for every $x\in \cM$ there are sequences $x_n\in \cM_n$ with $\lim_n F_n(x_n)=x$
 and $\lim_n f_n(x_n) = f(x)$.

We proceed deriving an  result linking equicontinuity and local uniform convergence.

\begin{Proposition} \label{ascoli2}  Assume  $\lim_n(\cM_n,d_n,F_n)=(\cM,d)$,  take  $f:(\cM,d)\to\re$ and continuous  functions
 $f_n:(\cM_n,d_n)\to\re$. The family $f_n$ locally uniformly $F_n$--converges to $f$ if and only if it is
 pointwise convergent and  $F_n$--locally equicontinuous. \end{Proposition}
 \begin{proof} The implication (local uniform convergence) $\Rightarrow$ (equicontinuity) has already
 been proved in Proposition
  \ref{ascoli1},  pointwise convergence can  also be  trivially  derived. This shows one half of the statement.
  Conversely, assume
  that a subsequence, still indexed by $n$, $y_n \in \cM_n$ satisfies $F_n(y_n)=x_0$, for some $x_0 \in \cM$,
  by pointwise convergence there
  is $z_n \in \cM_n$ with $\lim_n F_n(z_n)=x_0$, $\lim_n f_n(z_n)=f(x_0)$, therefore
  $x_n$, $z_n$ belong to $F_n^{-1}(\K)$,
  for a suitable compact subset $\K \subset \cM$ and any $n$, in addition
  \[ \md_n(y_n,z_n) \leq K \,  \md(F_n(y_n),F_n(z_n)) + K \, A_n \overset n \longrightarrow 0 \]
 so that,  by local equicontinuity
  \[ \lim_n |f_n(z_n)-f_n(y_n)| =0,\]
  which in the end implies
  \[ \lim_nf_n(y_n)=f(x_0),\]
  as desired.

 \end{proof}

 Next proposition put in relation local global and  uniform $F_n$--convergence.

 \begin{Proposition}\label{ascoli3} Assume  $\lim_n(\cM_n,d_n,F_n)=(\cM,d)$,  let $f_n:(\cM_n,d_n)\to\re$ be   continuous. If $f_n$
 locally uniformly $F_n$--converges to some function $f:(\cM,d)\to\re$ then
 \begin{equation}\label{ascoli31}
 \lim_n \sup_{x\in F_n^{-1}(\K)} \lv f_n(x) -f(F_n(x))\rv =0 \quad\hbox{for any compact subset $\K$ of $\cM$.}
\end{equation}
Conversely, if \eqref{ascoli31} holds and $f$ is continuous then  $f_n$ locally uniformly $F_n$--converges to $f$.
 \end{Proposition}

 \begin{proof} First, assume $f_n$ locally uniformly convergent to $f$. If \eqref{ascoli31} were not true, there should be a
 compact subset of $\K \subset \cM$, $\de >0$, and a (sub)sequence $y_n \in F_n^{-1}(\K)$ with
 \begin{equation}\label{ascoli32}
    |f_n(y_n)- f(F_n(y_n))| > \de
\end{equation}
 Since $F_n(y_n)$ converges, up to subsequences, to some $x \in \K$, we get  by local uniform convergence
 $\lim_n f_n(y_n)= f(x)$, being $f$ continuous by Proposition \ref{ascoli1}, we also have $\lim_n f(F_n(y_n))= f(x)$.
 These two limit relations are in contrast with \eqref{ascoli32}.

  Conversely, if \eqref{ascoli31} holds and a (sub)sequence $F_n(y_n) \in \cM_n$ converges to $x \in \cM$, then
  $y_n \in F_n^{-1}(\K)$, for some compact subset of $\cM$ and
  \begin{equation}\label{ascoli33}
    |f_n(y_n)- f(F_n(y_n))| \longrightarrow 0
\end{equation}
  the fact that  $f$ is continuous by assumption  implies  $\lim_n f(F_n(y_n))= f(x)$, this last relation,
   combined with \eqref{ascoli33}, yields
  \[ \lim_n f_n(y_n)= f(x).\]
\end{proof}


 \section{Homogenization in the maximal free abelian cover}\label{STA}
Write
 \begin{align*}
 H_\e(x,p) :&= \tH\big(x,\tfrac 1\e p\big). \\
 L_\e(x,v) :&= \max_{p\in T^*_xM}\left\{\, p\cdot v - H_\e(x,p)\, \right\}\\
 &= \max_{p\in T^*_xM} \left\{\frac p\e\cdot(v\e) - \tH\big(x,\tfrac p\e\big)\right\} \\
 &= L(x,\e v).
 \end{align*}
 The solution to the problem \eqref{vve1}--\eqref{vve2} is given by the Lax-Oleinik formula
 \begin{align*}
 v^\e(x,t)&=\inf \left\{\, f_\e(\ga(0))+\int_0^t L_\e(\ga,\dot\ga)\;dt\;\Big|\;
  \ga\in C^1([0,t],\tM),\,
 \ga(t)=x\;\right\} ,\\
 &=\inf_{\ga(t)=x}\left\{\; f_\e(\ga(0))+\int_0^t L(\ga,\e\,\dga)\;\right\}.
 \end{align*}

 Write $\eta:[0,\tfrac t\e]\to \tM$, $\eta(s):=\ga(\e s)$. Then
 \begin{align}
 \int_0^t L(\ga,\e\dga)\;dt &= \int_0^{\frac t\e} L(\eta(s),\deta(s))\;\e \, ds.
 \\
 v^\e(x,T)&=\inf \left\{\,f_\e(\eta(0))+
 \e\int_0^{\frac t\e} L(\eta,\deta)\; ds\,
 \Big|\;\eta\in C^1([0,\tfrac t\e],\tM), \;\,
 \eta\big(\tfrac t\e\big)=x\;
  \right\},
  \notag
  \\
 &=\inf _{y\in \tM} \left\{ f_\e(y)+\e\,\tphi\big(y,x,\tfrac t\e\big)\,\right\}
 \label{vef}
 \end{align}
 where
 \begin{align*}
 \tphi(y,x,t):&= \inf
 \left\{ \,\int_0^{t} L(\eta,\deta)\; ds\,
 \Big|\;\eta\in C^1\big([0,S],\tM\big), \;\,\eta(0)=y,\;
 \eta(t)=x\;
  \right\}.
 \end{align*}

 The solution to the limit problem~\eqref{uue1}--\eqref{uue2} is
 \begin{equation}\label{uf}
  u(h,t) = \inf_{q\in H_1(M,\re)}\left\{ f(q)+t\,\be\left(\tfrac{h-q}t\right)\right\},
 \end{equation}
 where $\be$ is Mather's minimal action functional~\eqref{beta}. Optimal elements for $u(h,t)$ do exists because $f$
 is continuous by Proposition \ref{ascoli1} and with linear growth, $\be$ superlinear.

\smallskip

 The proof is just to show that formula~\eqref{vef} converges to formula~\eqref{uf},
 using Mather's proposition \ref{MatP} below on the uniform convergence of
mean minimal  actions to the beta function. See \cite[Corollary on
page 181]{Mat5}, we have just slightly changed notations:

\smallskip

\begin{Proposition}\quad\label{MatP}
For every $A>0$, $\de>0$ there is $T_0>0$ such that
 \\
  if
 \quad $x,\, y\in \tM$, \quad $T\ge T_0$, \quad $\lV\tfrac {G(y)-G(x)}T\rV\le A$, \quad then
 $$
 \lv \tfrac 1T\, \tphi(x,y,T) - \be\left(\tfrac{G(y)-G(x)}T\right)\rv < \de.
 $$
\end{Proposition}

We recall that  the element $G(y)-G(x)$  of $H_1(M,\re)$ is
characterized by the relation
\begin{equation}\label{diffvect}
 \langle c_i, G(y)-G(x)\rangle = \oint_x^y \hatom_i,
 \end{equation}
 for any $i=1, \cdots, k$, where $c_1, \cdots, c_k$  is a basis of $H_1(M,\re)$, $\om_i$ are representatives in $c_i$
 and    $\hatom_i$ are the lifts of $\om_i$ to $\tM$, see \eqref{defG}. We deduce from Proposition \ref{MatP}

 \begin{Proposition}\label{math}
 Let $x_\e$, $y_\e$ be in $\tM$, for any $\e >0$, and $h$, $q$ in $H_1(M,\re)$. Let $t_\e$ be
  a sequence of positive times  converging to $t_0>0$. If $\lim_\e F_\e(x_\e)=h$,
   $\lim_\e F_\e(y_\e)=q$ then
\[  \lim_\e  \e  \,  \tphi(y_\e,x_\e,\tfrac {t_\e}\e)= t_0 \, \be\left( \tfrac{h-q}{t_0}\right ). \]
\end{Proposition}
\begin{proof} Let $a$, $b$  be constants estimating from below  and above, respectively, $t_\e$,
at least for $\e$ suitably small. We take $A$, $\de$, $T_0$ as in Proposition \ref{MatP}.
For $\e$ small we have
\[  \lV\tfrac { \e \, (G(x_\e)-G(y_\e))}a\rV\le A 
\quad\text{ and }\quad \tfrac  a\e > T_0.
\]
We derive in force of Proposition \ref{MatP},
\[
   \lv  \e  \,  \tphi(y_\e,x_\e, \tfrac {t_\e}\e) -  t_\e \,\be\left(\tfrac{F_\e(x_\e)-F_\e(y_\e)}
{t_\e}\right)\rv < b \, \de \quad\hbox{for  $\e$ small}.
\]
Therefore
\[ \lim_\e \lv  \e  \,  \tphi(y_\e,x_\e, \tfrac {t_\e}\e) -  t_\e \,\be\left(\tfrac{F_\e(x_\e)-F_\e(y_\e)}
{t_\e}\right)\rv =0,\]
and the assertion follows being $\be(\cdot)$ continuous.

\end{proof}
 \smallskip

\medskip

 \begin{Lemma} \label{lemma1}  Given a compact subset  $\K$ in $H_1(M,\re)$,  and a
  compact interval $I \subset ]0,+ \infty [$, there is
  a positive constant $C$
such that
\[ d_\e(x,y) \leq C\]
for $\e >0$ suitably small, any $ x \in F_\e^{-1}(\K)$, $ t \in I$  and $y \in
\tM$ realizing the minimum in the formula yielding $v^\e(x,t)$.
\end{Lemma}

\begin{proof} By linear growth assumption on $f$ we find
\[ f(h) \geq - A \, \|h\| - B \]
for any $ h \in H_1(M, \re)$ and suitable positive constants $A$
and $B$, by applying uniform convergence of $f_\e$ to $f$ we infer
\begin{equation}\label{lem10}
    f_\e(x) \geq - A \, \e \, \|G(x)\| - (B +1) \qquad x \in \tM
\end{equation}
for $\e$ suitably small and, in addition, we find a constant $Q$ with
\begin{equation}\label{lem100}
    f_\e(x) \leq Q \qquad\hbox{for  $\e$ small, $x \in  F_\e^{-1}(\K)$.}
\end{equation}
 Since by condition \eqref{c1} in the definition of spaces convergence
$\|G(x)\|$  and $d(x,y)$ are
infinite of the same order, as $d(x,x_0) \to + \infty$, for any
fixed $x_0$,   we deduce from \eqref{lem10}, up to adjusting constants $A$,
$B$\begin{equation}\label{lem11}
     f_\e(x) \geq - A \, d_\e(x,x_0) - (B +1) \qquad x \in \tM.
\end{equation}
By superlinearity of $L$ we are able to find, for any $M>0$, a
positive $N$ such that, taken  a pair $x$, $y$ of elements of
$\tM$ and any curve $\ga$ linking them in $[0,t]$, one has
\begin{equation}\label{lem12}
    \int_0^t L(\ga, \e \,\dot\ga)\;dt \geq M \,  d_\e(x,y) - N \, t.
\end{equation}
By \eqref{lem100}
\[ v^\e(x_0,t) \leq Q +  t \, \min_{y \in M} L(y,0) \qquad\hbox{for $\e$ small, $x_0 \in  F_\e^{-1}(\K)$.}\]
By combining  this last inequality \eqref{lem11}, \eqref{lem12}, and taking into account the very definition of $v^\e$,
  we get the assertion.
\end{proof}

\medskip

 \noindent {\bf Proof of theorem~\ref{TA}:}

 Let $(h_0,t_0)$ be in $H_1(M,\re)  \times ]0,+ \infty[$.  Let $(x_\e,t_\e)$ be  a (sub)sequence  in
 $\tM \times ]0,+ \infty[ $ with $\lim_\e F_\e(x_\e)= h_0$, $\lim_\e
 t_\e=t_0$, our task is to show
 \begin{equation}\label{main1}
    \lim_\e v^\e(x_\e,t_\e)= u(h_0,t_0).
\end{equation}
Assume  a subsequence $v^{\e_k}(x_{\e_k},t_{\e_k})$ to have limit
 and consider   $y_{\e_k}$   optimal for $v^{\e_k}(x_{\e_k},t_{\e_k})$,
then, according to Lemma \ref{lemma1},
 $d_{\e_k}(x_{\e_k},y_{\e_k}) < C$ for  $\e_k$ small,  which implies that
 $\|F_\e(x_{\e_k})-F_\e(y_{\e_k})\|$, and consequently  $\|h_0-F_\e(y_{\e_k}\|$ are
 bounded. We deduce that
   $F_{\e_k}(y_{\e_k})$
 converges, up to subsequences,
 to some element $q_0 \in H_1(M,\re)$. We  therefore get in force of  Proposition \ref{math}
 \[ \lim_{\e_k} \e_k \, \tphi(y_{\e_k},x_{\e_k},\tfrac {t_{\e_k}}{\e_k}) = t_0 \, \be( \tfrac{h_0-q_0}{t_0}).\]
 Exploiting the uniform $F_\e$--convergence of $f_\e$ to $f$, we further
 derive
 \[ \lim_{\e_k} f(y_{\e_k}) + \e_k \, \tphi(y_{\e_k},x_{\e_k},\tfrac {t_{\e_k}}{\e_k})
    = f(q_0) + t_0 \, \be( \tfrac{h_0-q_0}{t_0}) \geq u(h_0,t_0)\]
    and consequently
\begin{equation}\label{main2}
    \liminf_\e  v^\e(x_\e,t_\e)  \geq u(h_0,t_0).
\end{equation}
We denote by $\bar q$ an optimal element for $u(h_0,t_0)$, there
is a sequence $y_\e$ in $\tM$ with  $\lim_\e F_\e(y_\e)= \bar q$,
therefore, again by  Proposition \ref{math} and convergence of
$f_\e$ to $f$ we get
 \[ \limsup v^\e(x_\e,t_\e) \leq \lim_{\e} f_\e(y_\e) + \e \, \tphi(y_{\e},x_{\e},\tfrac {t_{\e}}{\e}) = f(\bar q)+ t_0 \,
 \be( \tfrac{h_0- \bar q}{t_0})=u(q_0,t_0),\]
 this limit relation together with \eqref{main2} gives
 \eqref{main1}.

 \qed

 \bigskip
 \bigskip

  \section{subcovers}\label{subcovers}

  \noindent
  {\bf Proof of Theorem~\ref{TB}:}

Consider the lifts to $\tM$ of the solutions
   $v^\e$ and the initial conditions $f_\e$, as well as the the lift to $H_1(M,\re)$ of $f$ :
  $$ \tv^\e(x,t) := v^\e(\pi(x),t),
  \qquad
  \tf_\e(x) = f_\e( \pi(x)),
\qquad
 \tf(h) = f( \ff (h)).$$
 Due to the fact that $\hH$  is the lift of an Hamiltonian $H$ defined in $T^*M$ and the differential structures
  of $\tM$, $\hM$ are the same,   $\tv^\e(x,t)$ is solution to \eqref{vve1} with initial datum $ \tf_\e(x)$.
Given $\e >0$ and   $x_0  \in \tM$,  we have
\begin{equation}\label{main3}
    |\tf_\e(x_0) - \tf(F_\e(x_0))|= |f_\e( \pi(x_0)) - f( \ff (F_\e(x_0))|= |f_\e( \pi(x_0)) - f(\hF_\e(\pi(x_0)))|,
\end{equation}
 in addition if $\lim_\e F_\e(x_\e)= h_0$ for some sequence $x_\e$ in $\tM$, $h_0 \in H_1(M,\re)$,  then
 \[ \lim_\e \hF(\pi(x_\e)) = \lim_\e \ff F_\e(x_\e) = \ff (h_0)\]
 then,  by the  uniform  $\hF_\e$--convergence of $f_\e$ to  $f$ as $\e \to 0$
 \[ \lim_\e \tf_\e(x_\e)= \lim_\e f_\e(\pi(x_\e))= f(\ff (h_0))= \tf(h_0) \]
 this limit relation together \eqref{main3} shows the uniform  $F_\e$--convergence
of  $\tf_\e$ to  $\tf$.
By Theorem~\ref{TA} we therefore get that $\tv^\e(x,t)$ locally uniformly $F_\e$--converges to
\[  \tu(q,t):= \inf_{h\in\H_1(M,\re)}\Big\{\,\tf(h) + t\be\left(\tfrac{q-h}t\right)\Big\} \qquad (q,t) \in H_1(M,\re) \times [0,+\infty [. \]
If $ z \in \ker \ff$ then
\begin{align*}
\tu(q+z,t) &= \inf_{h\in\H_1(M,\re)}\Big\{\,\tf(h) + t\be\left(\tfrac{q+z-h}t\right)\Big\}\\ &=
\inf_{h \in\H_1(M,\re)}\Big\{\,\tf(h-z) + t\be\left(\tfrac{q+z-h}t\right)\Big\} = \tu(q,t).
\end{align*}
From this we recognize that $\tu$ is the lift  to $H_1(M, \re)$ of a function $u$ defined in $\im \ff \times [0,+ \infty[$.
  Using \eqref{hbe}, we have for any $q_0 \in \im \ff$, $ h_0 \in H_1(M, \re)$ with $\ff(h_0)= q_0$
  \begin{align*}
  u(q_0,t) &= \inf_{z \in \ker \ff} \tu(h_0+z,t) = \inf_{h\in H_1(M, \re)}\Big\{\,\tf(h) +
  \inf_{z \in \ker \ff} t\be\left(\tfrac{h_0+z-h}t\right)\Big\}
  \\ &=\inf_{q \in \im \ff}\Big\{\, f(q) + t \hbe\left(\tfrac{q_0-q}t\right)\Big\},
\end{align*}
so that  $u$ is the solution to the limit problem \eqref{lim1}--\eqref{lim2}. It is left to investigate the convergence
 of $v^\e$ to $u$. We consider $(q_0,t_0) \in \im \ff \times ]0,+ \infty[$, a (sub)sequence
 $(y_\e,t_\e) \in \hM \times ]0,+ \infty[$
 with $\lim_\e \hF_\e(y_\e)=q_0$, $\lim_\e t_\e=t_0$; then   there is  a compact subset $\K$ of $\im \ff$ with
\[ \ff F_\e(x_\e) \in \K \qquad\hbox{for $\e$ small, $x_\e \in \pi^{-1}(y_\e)$.}\]
   We can select $\K_0 \subset  \ff^{-1}(\K)$, $\K_0$  compact in $H_1(M,\re)$, with $\ff(\K_0)= \K$,
 in addition, taking into
 account that $H_1(M,\Z) \cap \ker \ff$ is $B$--dense  in $\ker \ff$ and \eqref{ff}, we infer that there is a compact
 enlargement $\K_1$ of $\K_0$ such that
 \[ \{x_\e \in \pi^{-1}(y_\e) \mid F_\e(x_\e) \in \K_1\}\neq \emptyset \quad\hbox{for $\e$ small.}\]
Choosing  $F_\e(x_\e)$, for $x_\e$ in the above set,   we build up a sequence  converging, up to a subsequence,
to some $h_0 \in H_1(M,\re)$ with $\ff(h_0)=q_0$, then by local uniform $F_\e$--convergence of $\tv^\e$ to $\tu$ we get
\[  \lim_\e v^\e(y_\e,t_\e)= \lim_\e \tv^\e(x_\e,t_\e)= \tu(h_0,t_0) = u(q_0,t_0).\]
This fact shows the sought local uniform $\hF_\e$--convergence of $v^\e$ to
$u$ and concludes the proof.

\qed

%

   \def\cprime{$'$} \def\cprime{$'$} \def\cprime{$'$} \def\cprime{$'$}
\providecommand{\bysame}{\leavevmode\hbox to3em{\hrulefill}\thinspace}
\providecommand{\MR}{\relax\ifhmode\unskip\space\fi MR }
\providecommand{\MRhref}[2]{%
  \href{http://www.ams.org/mathscinet-getitem?mr=#1}{#2}
}
\providecommand{\href}[2]{#2}

\end{document}